\documentclass[12pt]{article}
\usepackage{pict2e}
\usepackage{graphicx}
\usepackage{subfigure}
\usepackage{graphicx}
\usepackage{color}
\usepackage{multirow}

\def\0{\emptyset}

\begin{document}
	\newtheorem{assump}{Assumption}[section]
	\newtheorem{claim}{Claim}[section]
	\newtheorem{theorem}{Theorem}[section]
	\newtheorem{corollary}[theorem]{Corollary}
	\newtheorem{definition}[theorem]{Definition}
	\newtheorem{conj}[theorem]{Conjecture}
	\newtheorem{question}[theorem]{Question}
	\newtheorem{lemma}[theorem]{Lemma}
	\newtheorem{prop}[theorem]{Proposition}
	\newtheorem{Observation}[theorem]{Observation}
	\newtheorem{Case}[theorem]{Case}
	\newenvironment{proof}{\noindent {\bf
			Proof.}}{\rule{2mm}{2mm}\par\medskip}
	\newcommand{\remark}{\medskip\par\noindent {\bf Remark.~~}}
	\newcommand{\pp}{{\it p.}}
	\newcommand{\de}{\em}

\title{\bf Burning numbers of $t$-unicyclic graphs }

\author{Ruiting Zhang, ~Yingying Yu ~and~ Huiqing Liu\footnote {Hubei Key Laboratory of Applied Mathematics, Faculty of Mathematics and Statistics, Hubei University, Wuhan 430062, PR China. Partially supported by NSFC under grant number 11971158, email: hql\_2008@163.com.}}

\date{}

\maketitle \baselineskip 17.1pt

\begin{abstract}
Given a graph $G$, the burning number of $G$ is the smallest integer $k$ for which there are vertices $x_1, x_2,\ldots,x_k$ such that $(x_1,x_2,\ldots,x_k)$ is a burning sequence of $G$.
It has been shown that the graph burning problem is NP-complete,  even for trees with maximum degree three, or linear forests. A $t$-unicyclic graph is a unicycle graph with exactly one vertex of degree greater than $2$. In this paper, we first present the bounds for the burning number of $t$-unicyclic graphs, and then use the burning numbers of linear forests with at most three components to determine the burning number of all $t$-unicyclic graphs for $t\le 2$.

\vskip 0.2cm


\vskip 0.2cm


\end{abstract}

\section{Introduction}%

Let $G=(V(G),E(G))$ be a simple graph, and we simply write $|V(G)|=|G|$ and $|E(G)|=||G||$. For any nonnegative integer $k$ and a vertex $u$, the \textit{$k$-th closed neighborhood of $u$}, denoted by $N_k[u]$, is the set of vertices whose distance from $u$ is at most $k$; we denote $N_1[u]$ simply by $N[u]$, and $d(u)=|N(u)|$ is the \textit{degree} of $u$. Let $\Delta(G)=\max\{d(u):u\in V(G)\}$. A vertex of degree $k$ is called a $k$-vertex. For two vertices $u$ and $v$ in $G$, the {\em distance between $u$ and $v$}, denoted by $d_G(u,v)$, and simply write $d(u,v)$, is the number of edges in a shortest path joining $u$ and $v$. The \textit{eccentricity} of a vertex $ v $ is the greatest distance from $v$ to any other vertices.

A tree is a connected acyclic graph. A caterpillar is a tree where deleting all vertices of
degree 1 leaves a path. Let $SP_s (x) (s\ge 3) $ be a
\textit{generalized star} which is a tree with exactly one $s$-vertex $x$ of degree more than 2, and a path
can be seen as an $SP_2(x)$. A \textit{rooted tree} $T(x)$ is a tree with one vertex $ x $ designated as the \textit{root}. The \textit{height} of a rooted tree $T(x)$ is the eccentricity of $x$. A \textit{rooted tree partition} of $ G $ is a collection of rooted trees that are subgraphs of $G$, with the property that the vertex sets of the trees partition $V(G)$.

A unicyclic graph $G$ is a connected graph with $||G||=|G|$. A {\em $t$-unicyclic graph} $U_t(x)$ ($t\ge 1$) is a unicyclic graph in which there exists exactly one $(t+2)$-vertex $x$ of degree more than 2. A cycle can be seen as a $0$-unicyclic graph. We denote a \textit{path/cycle/star} of order $n$ by $P_n/C_n/S_n$.

Let $G_1,G_2,\ldots,G_k$ be $k$-vertex disjoint graphs.
The union of graphs $ G_i ~(1\le i\le k) $, denoted by $ G_1 +\cdots+G_k $ , is the
graph with vertex set
$\bigcup_{i=1}^{k}V (G_i) $ and edge set
$\bigcup_{i=1}^{k}E(G_i)$. For any positive integer $k$, let $[k]$ denote the set $\{0,1,2,\cdots,k-1\}$.

 Bonato-Janssen-Roshanbin \cite{BoJR14} defined {\em graph burning} involving a graph process as follows. Initially, at time round $t=0$ all vertices are unburned. During each time round $t\ge 1$, one new unburned vertex is chosen to burn (if such a vertex exists):  burned vertex remains burned until the end of the process and its unburned neighbors becomes burned. The process ends when all vertices are burned. A graph is called {\em $k$-burnable} if it can be burned by at most $k$ rounds. The {\em burning number} of a graph $G$, denoted by $b(G)$, is the smallest integer $k$ such that $G$ is $k$-burnable. The vertices that are chosen to be burned are referred to as a {\em burning sequence}, and a shortest such sequence is called {\em optimal}. Note that any optimal burning sequences have length $b(G)$. Figure 1 illustrates an optimal burning sequence the path $C_4$ (resp., $C_5$) with vertices $\{v_1, v_2, v_3, v_4\}$ (resp.,  $\{v_1, v_2, v_3, v_4,v_5\}$).

\begin{center}
\setlength{\unitlength}{1mm}
\begin{picture}(100,30)
\thinlines

\put(5,10){\circle*{1}} \put(5,25){\circle*{1}}
\put(20,10){\circle*{1}} \put(20,25){\circle*{1}}

\put(5,10){{\circle{2.5}}}\put(20,25){\circle{2.5}}
\put(5,10){\line(1,0){15}}\put(5,25){\line(1,0){15}}
\put(5,10){\line(0,1){15}}\put(20,10){\line(0,1){15}}

\put(0,9){$v_{1}$}   \put(22,9){$v_{2}$}
\put(0,24){$v_{4}$}  \put(22,24){$v_{3}$}

\put(65,12){\circle*{1}}\put(65,23){\circle*{1}}
\put(76,7.5){\circle*{1}} \put(76,27.5){\circle*{1}}\put(86,17.5){\circle*{1}}

\put(65,12){\circle{2.5}}\put(65,23){\circle{2.5}}\put(86,17.5){\circle{2.5}}

\put(65,12){\line(0,1){11}} \put(65,12){\line(5,-2){11}}\put(65,23){\line(5,2){11}}

\put(86,17.5){\line(-1,1){10}}\put(86,17.5){\line(-1,-1){10}}

\put(59.5,10){$v_{1}$} \put(77,5){$v_{2}$}
\put(88,17){$v_{3}$} \put(74,29){$v_{4}$}  \put(59,22){$v_{5}$}

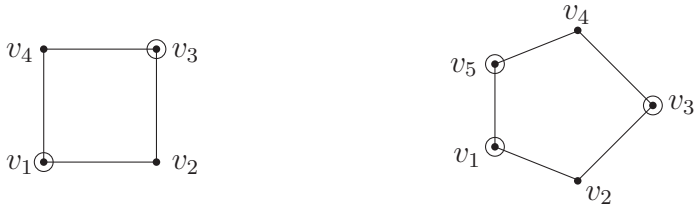
\put(-12,-1){Figure 1.~Burning $C_4$ and $C_5$ (the open circles represent burned vertices) }

\end{picture}
\end{center}


Burning can be viewed as a simplified model for the spread of social influence in a social network such as Facebook or Twitter  \cite{Kramer}, which is motivated by two other well-known models: the
competitive diffusion game \cite{AlonFe} and Firefighting \cite{Finbow}. It turned out that graph burning is related to several other graph theory problems such as graph bootstrap percolation
\cite{Balogh}, and graph domination \cite{Haynes}.
Recently, the graph burning problem attracted the attention of many researchers and some results are achieved, see
\cite{FiWi18,ReJa20,LaLu16,LiZH19}, \cite{MiPR17}-\cite{DeMo20}, \cite{SiTW18}). For more on graph burning and graph searching, see \cite{AnBo20} and \cite{BoPr17}. The graph burning problem is NP-complete,  even for trees with maximum degree three, or linear forests (i.e., a disjoin union of paths), see \cite{{BBJR17},BBJR18,Rosh16}. So it is interesting to compute the burning number of some special classes of graphs.


Roshanbin \cite{Rosh16} gave the following criterion for a graph to have burning number 2.

\begin{theorem}\cite{Rosh16}
\label{b=2}
A graph $G$ of order $n$ satisfies $b(G) =2$ if and only if $|G|\ge 2$ and $n-2\le \Delta(G)\le n-1$.
\end{theorem}

Bonato {\em at al.} in \cite{BoJR14} studied the burning number of paths and cycles.

\begin{theorem} \cite{BoJR14}
\label{Pn}
For a path $P_n$ or a cycle $C_n$, we have that $b(P_n)=b(C_n) =\lceil\sqrt{n}\rceil$. Moreover, if a graph $G$ has a Hamiltonian path, then $b(G) \le\lceil\sqrt{|G|}\rceil$.
\end{theorem}

A graph $G$ of order $n$ is {\em well-burnable} if $b(G)\le \lceil\sqrt{n}\rceil$. Bonato {\em at al.} \cite{BoJR14} conjectured that all connected graphs are well-burnable, and they showed that $P_n$ and $C_n$ are well-burnable.  As showed in \cite{LiHH20} and independently in \cite{HiTK}, the caterpillar is well-burnable. Bonato and Lidbetter \cite{BoLi18+} proved that the generalized star is also well-burnable. Moreover, Tan and Teh \cite{TaWe20} showed that the bound $ \lceil\sqrt{n}\rceil$ is tight for generalized stars.

\begin{theorem}
\cite{BoJR16,TaWe20}
\label{spider}
Let $G$ be a generalized star of order $n$, then $b(G)\le \lceil\sqrt{n}\rceil$.
\end{theorem}

Recently, Liu {\em et. al} \cite{LiZH20} determined the burning numbers of linear forests with at most three components.

\begin{theorem}\label{f2}\cite{LiZH20}
Let $G=P_{a_1}+P_{a_2}$ with $a_1\geq a_2\geq 1$, and let $J(t)=\{(t^2-2,2):t\ge 2$ is an integer$\}$, then
$$b(G)= \left\{
\begin{array}{ll}
\left\lceil\sqrt{a_1+a_2}\right\rceil+1, & \mbox{if~} (a_1,a_2)\in J(t) ;\\
\lceil\sqrt{a_1+a_2}\rceil, & \mbox{~}  otherwise.
\end{array}\right.
$$
\end{theorem}

Let $J^{1}=\{(a_1,a_2,a_3):(a_2,a_3)\in D_1$, $\sum_{i=1}^3a_i= t^2-3$ for some $t$\},

\hspace{0.74cm}$J^{2}=\{(a_1,a_2,a_3):(a_2,a_3)\in D_1\cup D_2$, $\sum_{i=1}^3a_i= t^2-2$ for some $t$\},

\hspace{0.74cm}$J^{3}=\{(a_1,a_2,a_3):(a_2,a_3)\in \cup_{i=1}^3D_i$ and $\sum_{i=1}^3a_i= t^2-1$ for some $t$\}$\cup \{(11,11,2)\}$,

\hspace{0.74cm}$J^{4}=\{(a_1,a_2,a_3):a_3=2$ or $(a_2,a_3)\in \cup_{i=1}^{4}D_i$, $\sum_{i=1}^3a_i= t^2$  for some $t$\}.

\hspace{0.74cm}$J^{5}=\{(13,11,1),(11,11,3),(22,13,1),(19,13,4),(17,13,6),(15,13,8),$

\hspace{1.78cm} $(13,13,10),(17,15,4),(15,15,6),(30,15,4),(28,15,6),(26,15,8),$

\hspace{1.78cm} $(19,15,15),(28,17,4),(26,17,6),(17,17,15),(26,19,4),(43,17,4),$

\hspace{1.78cm} $(41,17,6),(30,17,17),(41,19,4),(30,30,4),(58,19,4)\}$, \\
where $D_1=\{(2,2)\}$, $D_2=\{(3,2)\}$, $D_3=\{(1,1),(3,3),(4,2),(5,5)\}$ and $D_4=$ $\{(2,1),(4,1),(4,3),(4,4),(6,1),(6,4),(6,5),(6,6),(7,7),(8,4),(8,6),(10,4)\}$.

\begin{theorem}\label{f3}\cite{LiZH20}
Let $G=P_{a_1}+P_{a_2}+P_{a_3}$ with $a_1\geq a_2\geq a_3\geq1$. Then
$$b(G)= \left\{
\begin{array}{ll}
\lceil\sqrt{a_1+a_2+a_3}\rceil+1, & \mbox{if~} (a_1,a_2,a_3)\in J^{1}\cup J^{2}\cup J^{3}\cup J^{4}\cup J^{5};\\
\lceil\sqrt{a_1+a_2+a_3}\rceil, & \mbox{~}  otherwise.
\end{array}\right.
$$
\end{theorem}

Let $U_{g}^{a_1,\ldots, a_t}$ be a \textit{unicyclic graph} of order $n$ obtained from a cycle $C_{g}=v_1v_2\cdots v_gv_1$ by attaching $t$ paths of lengths $a_{1},a_{2},\ldots,a_{t} $ at the vertex $v_g$,  where $a_1\ge a_2\ge \cdots \ge a_{t}\geq 1$ (see Figue 2). Then $P_{a_i}$ ($1\le i\le t$) are called the {\em arms}. Denote $\mathcal U_{n,g}^{(t)}=\{U_{g}^{a_1,\ldots, a_t}:n=g+\sum_{i=1}^ta_i\}$ for $1\le t\le n-g$.

\begin{center}
\setlength{\unitlength}{1.0mm}
\begin{picture}(70,40)
\thinlines

\put(30,20){\circle*{2}}\put(34.4,9.4){\circle*{2}}
\put(34.4,30.6){\circle*{2}}\put(45,5){\circle*{2}}
\put(45,35){\circle*{2}}\put(55.6,9.4){\circle*{2}}
\put(55.6,30.6){\circle*{2}}\put(60,20){\circle*{2}}

\qbezier(30,20)(31,34)(45,35)\qbezier(30,20)(31,6)(45,5)
\qbezier(45,35)(59,34)(60,20)
\put(32,19){$v_g$}\put(36.5,9){$v_{g-1}$}
\put(31,33){$v_1$}\put(43,38){$v_2$}

\put(59.7,16.7){\circle*{0.75}}\put(58.9,14){\circle*{0.75}}
\put(57.8,11.3){\circle*{0.75}}\put(47.95,5.2){\circle*{0.75}}
\put(50.3,6){\circle*{0.75}}\put(53,7.2){\circle*{0.75}}

\put(25,15){\circle*{2}}\put(25,25){\circle*{2}}

\put(15,15){\circle*{0.75}}\put(15,20){\circle*{0.75}}
\put(15,25){\circle*{0.75}}

\put(15,5){\circle*{2}}\put(15,35){\circle*{2}}

\put(19,31){\circle*{0.75}} \put(21,29){\circle*{0.75}}
\put(20,30){\circle*{0.75}}

\put(20.3,10.3){\circle*{0.75}} \put(21.3,11.2){\circle*{0.75}}
\put(19.3,9.3){\circle*{0.75}}

\put(14.5,4.5){\line(1,1){4}}\put(14.5,35.5){\line(1,-1){4}} \put(30,20){\line(-1,1){8}} \put(30,20){\line(-1,-1){8}}
\put(6,19){$t~\left\{\rule{0mm}{17mm}\right.$}

\put(25,-3){Figue 2. $U_{g}^{a_1,\ldots, a_t}$} \put(45,19){$C_g$} \put(20,32){$P_{a_1}$} \put(20,7){$P_{a_t}$}

\end{picture}
\end{center}

In this paper, we show that the burning number of all unicyclic graphs in $\mathcal U_{n,g}^{(t)}$ of order $n=q^2+r$ with $1\le t\le 2$ is either $q$ or $q+1$, where $1\leq r\leq 2q+1$. Furthermore, we find all the sufficient conditions for all unicyclic graphs in $\mathcal U_{n,g}^{(t)}$ ($1\le t\le 2$) to have burning number $q$ or $q+1$, respectively.

\section{Preliminaries}


We first cite the following results about graph burning.

\begin{theorem}  \cite{BoJR14,BoJR16,Rosh16}
\label{b-sequ}
In a graph $G$, the sequence $(x_1, x_2, \ldots, x_k)$ forms
a burning sequence if and only if, for each pair $i$
and $j$, with $1 \leq i <j \leq  k$, $d(x_i, x_j)\geq  j -i$, and the following set equation holds:
$N_{k-1}[x_1]\cup N_{k-2}[x_2]\cup\cdots\cup N_{1}[x_{k-1}]\cup N_{0}[x_k]=V(G)$.
\end{theorem}

\begin{theorem}\cite{BoJR16}
\label{tree-partition}
Burning a graph $G$ in $k$ steps is equivalent to finding a rooted tree partition $\{T_1,T_2,\cdots, T_k\}$, with heights at most $(k-1), (k-2), . . . , 0$, respectively, such
that for every $1 \le i, j\le k$, the distance between the roots of $T_i$ and $T_j$ is at least $|i-j|$.
\end{theorem}

\begin{lemma}\cite{BoJR16}
	\label{spanning}
	For a graph $G$, we have $b(G)=\min\{b(T): T~is~a~spanning~tree~ of~G\}$.
\end{lemma}

A subgraph $H$ of a graph $G$ is called an {\em isometric subgraph} if for every
pair of vertices $u$ and $v$ in $H$, we have that $d_H(u, v) = d_G(u, v)$.

\begin{lemma}\cite{BoJR16,Rosh16}
\label{isometric}
For any isometric subtree $H$ of a graph $G$, we have $b(H)\leq
b(G)$.
\end{lemma}

The following assertions will be used in the proof of our main results.

\begin{prop}
\label{Sp4}
Let $w\in V(SP_s(x))$ ($s\ge 3$) with $d(x,w)=t$. If the height of $SP_s(x)$ with $w$ as the root is at most $i$, then $|SP_s(x)|\le  2i+1+(s-2)(i-t)$. Moreover,

(i) if $|SP_s(x)|=si+1$, then $x=w$ and $SP_s(x)-x=sP_i$;

(ii) if $|SP_s(x)|=si$, then $x=w$ and $SP_s(x)-x=(s-1)P_i+P_{i-1}$,  or $d(x,w)=1$, $s=3$ and $SP_s(x)-\{x,w\}=P_i+2P_{i-1}$;

(iii) if $|SP_s(x)|=si-1$, then either $x=w$ and $SP_s(x)-x\in\{(s-1)P_{i}+P_{i-2},(s-2)P_i+2P_{i-1}\}$, or $d(x,w)=2$, $s=3$ and $SP_s(x)-\{x,y,w\}=P_i+2P_{i-2}$ where $xyw$ is a path of length 2, or $d(x,w)=1$, $s=4$ and $SP_s(x)-\{x,w\}=P_i+3P_{i-1}$, or $d(x,w)=1$, $s=3$ and $SP_s(x)-\{x,w\}=P_i+P_{i-1}+P_{i-2}$.

\end{prop}

\begin{proof} Let $x_j$ be a vertex of degree 1 of $SP_s(x)$, where $1\le j\le s$.
Note that $d(w,z)\le i$ for any $z\in V(SP_s(x))$. Hence $$ |SP_s(x)|\le i+t+1+(s-1)(i-t)=2i+1+(s-2)(i-t).\eqno(*)$$

(i) If $|SP_s(x)|=si+1$, then by ($*$), $t=0$, and then $d(w,x_j)=i$ for $1\le j\le s$, that is, $x=w$ and $SP_s(x)-x= sP_i$.

(ii) If $|SP_s(x)|=si$, then by ($*$), $(s-2)t\le 1$ which implies $t\le 1$ as $s\ge 3$, moreover, if $t=1$, then $s=3$, that is $d(x,w)=1$, and then $SP_s(x)-\{x,w\}=P_i+2P_{i-1}$.
If $t=0$, that is $x=w$, then there exists some $j_0$ with $1\le j_0\le s$ such that $d(w,x_{j_0})= i-1$ and $d(w,x_j)= i$ for $1\le j\le s$ and $j\not=j_0$, thus $SP_s(x)-x=(s-1)P_{i}+P_{i-1}$.

(iii) If $|SP_s(x)|=si-1$, then by ($*$), $(s-2)t\le 2$ which implies $t\le 2$ as $s\ge 3$. If $t=2$, then $s=3$, and then $d(x,w)=2$ and $SP_s(x)-\{x,y,w\}=P_i+2P_{i-2}$, where $xyw$ is a path of length 2. If $t=1$, then $3\le s\le 4$, that is $d(x,w)=1$, and then $SP_s(x)-\{x,w\}=P_i+3P_{i-1}$ for $s=4$, and $SP_s(x)-\{x,w\}=P_i+P_{i-1}+P_{i-2}$ for $s=3$.
If $t=0$, then $x=w$ and $SP_s(x)-x\in\{(s-1)P_{i}+P_{i-2},(s-2)P_i+2P_{i-1}\}$.
\end{proof}

For any $G\in \mathcal U_{n,g}^{(t)}$ with $\Delta(G)=d(v_g)=t+2$ and $d(v)\le 2$ for any $v\neq v_g$, we always suppose that $(x_1, x_2,\ldots, x_k)$ is an optimal burning sequence for $G$. By Theorem \ref {tree-partition}, there exists a rooted tree partition $\{T_1(x_1), T_2(x_2), \ldots, T_k(x_k)\}$ of $G$ with heights at most $(k-1), (k -2), \ldots, 0$, respectively. Then $T_i(x_i)= P_{l_i}$ with $l_i\le 2(k-i)+1$, or $T_i(x_i)= SP_s(v_g)$ with $3\le s\le t+2$. Assume $v_g\in V(T_{k-i_0}(x_{k-i_0}))$ for some $ i_0\in [k]$. Then by Proposition \ref{Sp4}, $|T_{k-i_0}(x_{k-i_0})|\le 2i_0+1+(s-2)(i_0-p_{i_0})$ with $s\le t+2$ and $d(v_g,x_{k-i_0})=p_{i_0}\ge 0$, and then $|T_{k-i}(x_{k-i})|\le 2i+1$ for $i\in [k]\backslash\{i_0\}$.
Hence,
\begin{eqnarray*}~~~~~n&=&\sum_{i=1}^k|T_i(x_i)|
\le \sum_{i=0}^{k-1}(2i+1)-(2i_0+1)+2i_0+1+(s-2)(i_0-p_{i_0})\\
&=& k^2+(s-2)(i_0-p_{i_0})\le k^2+t(k-1),~~~~~~~~~~~~~~~~~~~~~~~~~~~~~~~~~~~~~~~~~~~(1)\end{eqnarray*}
where the last inequality follows from $p_{i_0}\ge 0$, $s\le t+2$ and $i_0\le k-1$.

Note that $C_g$ is the unique cycle of $G$, and hence $H:=G-e$ is a generalized star for any $e\in E(C_g)$. By Lemma \ref{spanning} and Theorem \ref{spider}, $b(G)\leq b(H)\leq \left\lceil \sqrt{n}\right\rceil$. So, by (1),  we have the following result.

\begin{theorem}\label{aq} For any $G\in \mathcal U_{n,g}^{(t)}$, we have
$$\left\lceil\sqrt{n+\frac{t^2+4t}{4}}-\frac{t}{2}\right\rceil\le b(G)\le \lceil\sqrt{n}\rceil.\eqno(2)$$
\end{theorem}

If $t\le 2$, then by (1), $b(G)\ge  \left\lceil\sqrt{n+3}\right\rceil-1\ge \left\lceil\sqrt{n}\right\rceil-1.$ Therefore, we have the following result.

\begin{corollary}\label{unicyclic3}
Let $G\in \mathcal U_{n,g}^{(t)}$ with $n=q^2+r$ ($1\le r\le 2q+1$). If $1\le t\le 2$, then $q\leq b(G)\leq q+1$.
\end{corollary}

\section{Burning number of unicyclic graphs in $\mathcal U_{n,g}^{(1)}$}

In this section, $\mathcal U_{n,g}^{(1)}=\{U_{g}^{n-g}\}$. Let $V(U_{g}^{n-g})=\{v_1,\cdots, v_g, w_1,\cdots w_{n-g}\}$ with $w_0=v_g$ and $d(v_g,w_i)=i$ for $0\le i\le n-g$. By Theorem \ref{b=2}, if $ 4\le n\le 5 $, $ b(G)=2 $ for any $ G\in \mathcal U_{n,g}^{(1)} $. So we can assume $n\ge 6$. Then $b(G)\ge 3$ for $ G\in \mathcal U_{n,g}^{(1)} $.

\subsection{Unicyclic graphs in $\mathcal U_{n,g}^{(1)}$ with burning number $q$}

In this subsection, we will give some sufficient conditions for the unicyclic graphs in $\mathcal U_{n,g}^{(1)}$ to have burning number $q\ge 3$. Recall that $d(v_g) =3$ and $n=q^2+r$ with $1\leq r\leq 2q+1$. Denote $\mathcal A=\{(2q+1,q^{2}-q-2),(q^{2}-2,q+1)\}$.

\begin{lemma}\label{q1}
Suppose that $1\le r\le q-1$. If  $2r+1\le g\le q^{2}$ and $(g,n-g)\notin \mathcal A$, then $b(U_{g}^{n-g})\le q$.
\end{lemma}

\begin{proof} Let $G_w:=U_{g}^{n-g}-N_{q-1}[w]$ for $w\in V(U_{g}^{n-g})$. If $b(G_w)\le q-1$, then $b(U_{g}^{n-g})\le q$ by Theorem \ref{b-sequ}. Hence we can assume $b(G_w)\ge q$ for any $w\in V(U_{g}^{n-g})$.

Let $g':=\lfloor\frac{g-1}{2}\rfloor$. Then $r\le g'$ as $g\ge 2r+1$. If $g\le 2q-1$, then we can assume $n-g'\ge 2q$. Otherwise, $V(U_{g}^{n-g})\subseteq N_{q-1}[w_{q-g+g'}]$, then $b(U_{g}^{n-g})\le q$ by Theorem \ref{b-sequ}. So $G_{w_{q-g+g'}}= P_{n-g'-2q+1}$ with $n-g'-2q+1= q^2+r-g'-2q+1\le (q-1)^2$ as $r\le g'$, and thus, by Theorem \ref{Pn}, $b(G_{w_{q-g+g'}})\le q-1$, a contradiction. Therefore, $g\ge 2q$.

If $n-g\le q-1$, then $G_{v_g}= P_{g-2q+1}$ and $|G_{v_g}|\le (q-1)^2$ as $g\le q^2$, and then by Theorem \ref{Pn}, $b(G_{v_g})\le q-1$, a contradiction. So $n-g\ge q$, and then $G_{v_g}= P_{g-2q+1}+ P_{n-g-q+1}$ with $|G_{v_g}|= q^2+r-3q+2\le(q-1)^2$ as $r\le q-1$. Note that $(g,n-g)\notin \mathcal A$, and hence $2\notin \{n-g-q+1,g-2q+1\}$. By Theorem~\ref{f2}, $b(G_{v_g})=q-1$, a contradiction.
\end{proof}

By Corollary \ref{unicyclic3} and Lemma \ref{q1}, we have the following result.

\begin{theorem}
\label{b=q,}
Suppose that $U_{g}^{n-g}$ is a unicyclic graph  of order $n=q^2+r$ with $1\le r\le q-1$. If $2r+1\le g\le q^{2}$ and $(g,n-g)\notin \mathcal A$,  then $b(U_{g}^{n-g})= \left\lceil\sqrt{n}\right\rceil-1=q$.
\end{theorem}

\subsection{Unicyclic graphs in $\mathcal U_{n,g}^{(1)}$ with burning number $q+1$}

 In this subsection, we will characterize the graphs in $\mathcal U_{n,g}^{(1)}$ with burning number $q+1$. Recall that $\mathcal U_{n,g}^{(1)}=\{U_{g}^{n-g}\}$ and $(x_1, x_2,\ldots, x_k)$ is an optimal burning sequence for $U_{g}^{n-g}$. By Corollary \ref{unicyclic3}, we only need to give some sufficient conditions for $U_{g}^{n-g}$ to satisfy  $k\ge q+1$.

\begin{lemma}\label{q+1,}
If $ r\ge q$, then $k\ge q+1$.
\end{lemma}

\begin{proof}
Since $ r\ge q$, we have $$n+\frac{5}{4}=q^2+r+\frac{5}{4}\geq q^2+q+\frac{5}{4}=(q+\frac{1}{2})^2+1.$$
Then $\sqrt{n+\frac{5}{4}}>q +\frac{1}{2}$, i.e.,
$\left\lceil\sqrt{n+\frac{5}{4}}-\frac{1}{2}\right\rceil\ge q+1$, and so, by (2), $k\ge q+1$.
\end{proof}

\begin{lemma}\label{q+1,1}
If $g\le 2r$, then $k\ge q+1$.
\end{lemma}

\begin{proof}
Note that $H:=P_{n-\lfloor\frac{g-1}{2}\rfloor}$ is an isometric subtree of $U_{g}^{n-g}$. Since $g\le 2r$, we have $\lfloor\frac{g-1}{2}\rfloor\le r-1$. Then $|H|=q^2+r-\lfloor\frac{g-1}{2}\rfloor\ge q^2+1$, and hence, by Theorem \ref{Pn}, $b(H)\ge q+1$. By Lemma \ref{isometric}, we have $k\ge b(H)\ge q+1$.
\end{proof}

\begin{lemma}\label{q+1,2}
If $g\geq q^{2}+1$, then $k\ge q+1$.
\end{lemma}

\begin{proof} Since $g\geq q^{2}+1$, $n-g=q^2+r-g\leq r-1$.
Recall that $v_g\in V(T_{k-i_0}(x_{k-i_0}))$. If $x_{k-i_0}\in V(C_g)$, then $|T_{k-i_0}(x_{k-i_0})|\le 2i_{0}+1+(n-g)$; and if $x_{k-i_0}\notin V(C_g)$, then $|T_{k-i_0}(x_{k-i_0})|\le 2(i_{0}-p_{i_{0}})+1+(n-g)\le 2i_{0}+1+(n-g)$. Hence, by (1),
$
n=q^2+r\le k^2+n-g\le k^2+r-1,
$
that is $k^{2}\geq q^2+1$. Thus $k\ge q+1$.
\end{proof}

\begin{lemma}\label{q+1,3}
If $(g,n-g)\in \mathcal A$, then $k\ge q+1$.
\end{lemma}

\begin{proof} Since $(g,n-g)\in \mathcal A$, $g\in \{2q+1,q^2-2\}$ and $n=q^2+q-1$. If $n\le k^2+k-2$, then $k\ge q+1$ clearly. So we can assume $n\ge k^2+k-1$. Then by (1), $n=q^2+q-1= k^2+k-1$, which implies $k=q$, $i_0=k-1=q-1$, $d(x_{k-i_0},v_g)=0$ (that is $x_1=v_g$), $|T_{k-i_0}(x_{k-i_0})|=3(q-1)+1$ and $|T_{k-i}(x_{k-i})|=2i+1$ for all $i\not=q-1$. Hence, by Proposition \ref{Sp4}, $T_1(x_1)-v_g=3P_{q-1}$, and thus, $U_{g}^{n-g}-N_{q-1}[v_g]= P_{g-2q+1}+ P_{n-g-q+1}$. Note that $g\in \{2q+1,q^2-2\}$, and hence $2\in \{g-2q+1,n-g-q+1\}$, which is a contradiction with $|T_{k-i}(x_{k-i})|=2i+1$ for all $i\not=q-1$. Hence $k\geq q+1$.
\end{proof}

\vskip .2cm
By Corollary \ref{unicyclic3} and Lemmas \ref{q+1,}-\ref{q+1,3}, we have the following result.

\begin{theorem}
\label{b=q+1,}
Suppose that $U_{g}^{n-g}$ is a unicyclic graph  of order $n=q^2+r$. If  $ q\le r\le 2q+1$, or $g\geq q^{2}+1$, or $g\le 2r$, or $(g,n-g)\in \mathcal A$, then $b(U_{g}^{n-g})=q+1$.
\end{theorem}


\noindent{\bf Note 1. } By Theorems \ref{b=q,} and \ref{b=q+1,}, $b(U_{g}^{n-g})$ are listed in Table 1.

\begin{table}[!hbp]
\centering
\vskip.2cm
\begin{tabular}{|c|c|c|}
\hline
\hline
$r$ & {$g$}  & $b(U_{g}^{n-g})$ \\
\hline

\multirow{2}{*}{$1\le r\le q-1$} & $g\geq q^{2}+1~or~g\le 2r~or~(g,n-g)\in \mathcal A$ &$q+1$\\\cline{2-3}
\multirow{2}{*}{} & $2r+1\le g\le q^{2}~and~(g,n-g)\notin \mathcal A$ &$q$\\

\hline
$q\le r\le 2q+1$ & {}&$q+1$\\
\hline

\end{tabular}
\caption{Burning number of $U_{g}^{n-g}$}
\end{table}

\section{Burning number of unicyclic graphs in $\mathcal U_{n,g}^{(2)}$}

Recall that $\mathcal U_{n,g}^{(2)}=\{U_{g}^{a_1,a_2}:a_1+a_2=n-g\}$ with $a_1\ge a_2$, and $\Delta(G)=d(v_g) =4$ for $G\in \mathcal U_{n,g}^{(2)}$. By Theorem \ref{b=2}, $b(G)=2$ if $G\in\{\mathcal U_{5,g}^{(2)},\mathcal U_{6,g}^{(2)}\} $. Hence we may assume $n\ge 7$ and $b(G)\ge 3$ for all $G\in \mathcal U_{n,g}^{(2)}$. Let $V(U_{g}^{a_1,a_2})=\{v_1,\ldots, v_g,u_1,\ldots,u_{a_1},w_1,\ldots w_{a_2}\}$ with $d(v_g,u_i)=i$ for $0\le i\le a_1$ and $d(v_g,w_j)=j$ for $0\le j \le a_2$, where $u_0=w_0=v_g$. Let $n=q^2+r$ with $1\leq r\leq 2q+1$. Then $\lceil\sqrt{n}\rceil=q+1$.

\subsection{Unicyclic graphs in $\mathcal U_{n,g}^{(2)}$ with burning number $q$}

In this subsection,  we will give some sufficient conditions for the unicyclic graphs in $\mathcal U_{n,g}^{(2)}$ to have burning number $q$. By Theorem \ref{b-sequ}, we only need to show that $b(U_{g}^{a_{1},a_{2}})\le q$, or find a burning sequence $(x_1,x_2,\cdots,x_q)$ so that $V(U_{g}^{a_{1},a_{2}})\subseteq N_{q-1}[x_1]\cup\cdots\cup N_1[x_{q-1}]\cup N_0[x_q]$. Denote $\mathcal B:=\{(2q-2,q^2-q-2,q+1),(2q-1,q^2-q-2,q+1)\}$.

\vskip .2cm

Let $v\in V(U_{g}^{a_{1},a_{2}})$, and let $H(v):=U_{g}^{a_{1},a_{2}}-N_{q-1}[v]$. If $b(H(v))\le q-1$, then $b(U_{g}^{a_{1},a_{2}})\le q$ by Theorem \ref{b-sequ}. Hence we can assume $b(H(v))\ge q$ for any $v\in V(U_{g}^{a_{1},a_{2}})$.

\begin{lemma}\label{q2}
	Suppose that $1\le r\le 2q-2$. If $r+1\le g\le 2q-1$, $a_{1}\le q^{2}-\lfloor\frac{g}{2}\rfloor-1$ and $(g,a_{1},a_{2})\notin \mathcal B$, then $b(U_{g}^{a_{1},a_{2}})\le q$.
\end{lemma}

\begin{proof}  If $a_2\le \lfloor\frac{g}{2}\rfloor$, then $H({u_{q-1-\lfloor\frac{g}{2}\rfloor}})= P_{a_1-2(q-1)+\lfloor\frac{g}{2}\rfloor}$ and $|H({u_{q-1-\lfloor\frac{g}{2}\rfloor}})|\le q^{2}-\lfloor\frac{g}{2}\rfloor-1-2q+2+\lfloor\frac{g}{2}\rfloor=(q-1)^2$, and then by Theorem~\ref{Pn}, $b(H(u_{q-1-\lfloor\frac{g}{2}\rfloor}))\leq q-1$, a contradiction.
	
So $a_{2}\geq \lfloor\frac{g}{2}\rfloor+1$. Then $H(u_{q-1-\lfloor\frac{g}{2}\rfloor})=P_{a_{1}'}+ P_{a_{2}'}$ with $a_1'=a_1-2(q-1)+\lfloor\frac{g}{2}\rfloor$, $a_2'=a_2-\lfloor\frac{g}{2}\rfloor$ and $a_{1}'+a_{2}'=q^{2}+r-g-(2q-2)\le(q-1)^2$ as $r+1\le g$. By our assumption and Theorem~\ref{f2}, we have $\min\{a_1',a_2'\}= 2$, $g=r+1$ and $a_{1}'+a_{2}'=(q-1)^2$.
	
	If $a_1'=2$, then $a_{2}'=q^2-2q-1$, i.e., $a_1=2q-\lfloor\frac{g}{2}\rfloor$ and $a_2=q^2-2q-1+\lfloor\frac{g}{2}\rfloor$. Since $a_1\ge a_2$, we have $q^2-4q-1+2\lfloor\frac{g}{2}\rfloor\le 0$, which implies $q=3$ and $g\le 5$. Since $(g,a_{1},a_{2})\notin \mathcal B$, i.e., $(g,a_{1},a_{2})\notin \{(4,4,4),(5,4,4)\}$ for $q=3$. Hence $g=3$, and then $n=11$, $a_1=5$ and $a_2=3$. Then $H(w_1)= P_4$, and then by Theorem~\ref{Pn}, $b(H(w_1))=2=q-1$, a contradiction.
	
	If $a_2'=2$, then $a_1=q^2-3-\lfloor\frac{g}{2}\rfloor$ and $a_2=\lfloor\frac{g}{2}\rfloor+2$. If $g\le 2q-3$, then $g-1\le 2(2q-4-\lfloor\frac{g}{2}\rfloor)$, and then $H(w_{\lfloor\frac{g}{2}\rfloor-q+3})= P_{a_{1}''}$ with $a_1''=a_1-(2q-4-\lfloor\frac{g}{2}\rfloor)=(q-1)^2$. By Theorem~\ref{Pn}, $b(H(w_{\lfloor\frac{g}{2}\rfloor-q+3}))=q-1$, a contradiction. Therefore, $2q-2\le g\le 2q-1$, furthermore, $a_1=q^2-q-2$ and $a_2=q+1$, a contradiction with $(g,a_{1},a_{2})\notin \mathcal B$.
\end{proof}

Let $B_1:=\{(q^2-2,q,q),(q^2-6,q+2,q+2),(q^2-10,q+4,q+4),(q^2-3,q+1,q),$

\hspace{1.98cm}$(q^2-5,q+3,q),(q^2-7,q+3,q+2),(q^2-8,q+3,q+3),(q^2-7,q+5,q),$

\hspace{1.98cm}$(q^2-10,q+5,q+3),(q^2-11,q+5,q+4),(q^2-12,q+5,q+5),(q^2-14$,

\hspace{1.98cm}$q+6,q+6),(q^2-12,q+7,q+3),(q^2-14,q+7,q+5),(q^2-14,q+9,q+3)\}$.

\hspace{0.68cm}$B_2:=\{(2q,q^2-q-2,q),(2q+2,q^2-q-6,q+2),(2q+4,q^2-q-10,q+4),(2q+1,$

\hspace{1.98cm}$q^2-q-3,q),(2q+3,q^2-q-5,q),(2q+3,q^2-q-7,q+2),(2q+3,q^2-q-8,$

\hspace{1.98cm}$q+3),(2q+5,q^2-q-7,q),(2q+5,q^2-q-10,q+3),(2q+5,q^2-q-11,$

\hspace{1.98cm}$q+4),(2q+5,q^2-q-12,q+5),(2q+6,q^2-q-14,q+6),(2q+7,$

\hspace{1.98cm}$q^2-q-12,q+3),(2q+7,q^2-q-14,q+5),(2q+9,q^2-q-14,q+3)\}$.\\

\hspace{0.68cm}$B_3:=\{(2q,q^2-q-3,q+1),(2q,q^2-q-5,q+3),(2q+2,q^2-q-7,q+3),$

\hspace{1.98cm}$(2q,q^2-q-7,q+5),(2q+3,q^2-q-10,q+5),(2q+4,q^2-q-11,q+5),(2q$

\hspace{1.98cm}$+3,q^2-q-12,q+7),(2q+5,q^2-q-14,q+7),(2q+3,q^2-q-14,q+9)\}$.

\hspace{0.68cm}$B_4:=\{(24,16,6),(22,16,8),(35,19,7),(32,19,10),(30,19,12),(28,19,14)$,

\hspace{1.98cm}$(26,19,16),(30,21,10),(28,21,12),(45,22,11),(43,22,13),(41,22,15)$,

\hspace{1.98cm}$(34,22,22),(43,24,11),(41,24,13),(32,24,22),(41,26,11),(60,25,12)$,

\hspace{1.98cm}$(58,25,14),(47,25,25),(58,27,12),(47,38,12),(77,28,13)\}$.

\hspace{0.68cm}$B_5:=\{(22,18,6),(22,16,8),(26,28,7),(26,25,10),(26,23,12),(26,21,14)$,

\hspace{1.98cm}$(26,19,16),(28,23,10),(28,21,12),(30,37,11),(30,35,13),(30,33,15)$,

\hspace{1.98cm}$(30,26,22),(32,35,11),(32,33,13),(32,24,22),(34,33,11),(34,51,12)$,

\hspace{1.98cm}$(34,49,14),(34,38,25),(36,49,12),(47,38,12),(38,67,13)\}$.

\hspace{0.68cm}$B_6:=\{(12,18,16),(14,16,16),(14,28,19),(17,25,19),(19,23,19),(21,21,19)$,

\hspace{1.98cm}$(23,19,19),(17,23,21),(19,21,21),(19,37,22),(21,35,22),(23,33,22)$,

\hspace{1.98cm}$(30,26,22),(19,35,24),(21,33,24),(30,24,24),(19,33,26),(21,51,25)$,

\hspace{1.98cm}$(23,49,25),(34,38,25),(21,49,27),(21,38,38),(23,67,28)\}$.

Denote $\mathcal C_1=\{(2q+1,q^{2}-q-2,r-q+1),(q^{2}-2,q+1,r-q+1)\}$;

\hspace{1.38cm}$\mathcal C_2=\{(q^{2}-7,q+2,q+1),(2q+1,q^{2}-q-6,q+1),(2q+1,q^{2}-q-7,q+2)\}$;

\hspace{1.38cm}$\mathcal C_3=\{(q^2-3,q,q),(q^2-5,q+1,q+1),(q^2-6,q+2,q+1),(q^2-7,q+2,q+2)$,

\hspace{2.38cm}$(q^2-7,q+3,q+1),(q^2-11,q+4,q+4),(2q+4,q^2-q-11,q+4),$

\hspace{2.38cm}$(2q+1,q^2-q-5,q+1),(2q+1,q^2-q-6,q+2),(2q+1,q^2-q-7,q+3),$

\hspace{2.38cm}$(2q+2,q^2-q-6,q+1),(2q+2,q^2-q-7,q+2),(2q+3,q^2-q-7,q+1),$

\hspace{2.38cm}$(2q,q^2-q-3,q),(22,16,7),(13,16,16)\}$;

\hspace{1.38cm}$\mathcal C_{4}=\{(g,a_{1},a_{2}):n=q^{2}+2q-2,g-q\geq a_{1}\geq a_{2}=q+1\}\cup B_{1}\cup B_{4}$;

\hspace{1.38cm}$\mathcal C_{5}=\{(g,a_{1},a_{2}):n=q^{2}+2q-2,a_{1}\geq g-q\geq a_{2}=q+1\}\cup  B_{2}\cup B_{5}$;

\hspace{1.38cm}$\mathcal C_{6}=\{(g,a_{1},a_{2}):n=q^{2}+2q-2,a_{1}\geq a_{2}\geq g-q=q+1\}\cup B_{3}\cup B_{6}$.

\vskip .2cm

If $g\ge 2q+2$, $a_1\ge a_2\ge q$, we let $\{m_1,m_2,m_3\}=\{g-2q+1,a_1-q+1,a_2-q+1\}$ with $m_1\ge m_2\ge m_3\ge 1$, then $m_3\not=a_1-q+1$ as $a_{1}\geq a_{2}$.

\begin{lemma}
\label{q4}
Suppose that $1\le r\le 2q-2$. If  $2q\le g\le q^{2}$, $a_{2}\geq r-q+1$ and $(g,a_{1},a_{2})\notin \bigcup_{i=1}^6 \mathcal C_i$, then $b(U_{g}^{a_{1},a_{2}})\leq q$.
\end{lemma}

\begin{proof} If $a_{2}\le q-1$, then $a_{1}\geq q$ (otherwise, $H(v_g)= P_{g-2q+1}$ with $|H(v_g)|\le (q-1)^{2}$ as $2q\le g\le q^{2}$, then by Theorem~\ref{Pn}, $b(H(v_g))\leq q-1$, a contradiction). So $H(v_g)= P_{g-2q+1}+ P_{a_{1}-q+1}$ with $|H(v_g)|=q^{2}+r-3(q-1)-1-a_2\le (q-1)^{2}$ as $a_2\ge r-q+1$. By our assumption and Theorem~\ref{f2}, $a_2=r-q+1$ and $2\in \{g-2q+1,a_{1}-q+1\}$, which implies $(g,a_{1},a_{2})\in \mathcal C_1$, a contradiction. Thus, $a_{2}\geq q$. Then $H(v_g) =P_{m_1}+ P_{m_2}+ P_{m_3}$ with $|H(v_g)|=q^2+r-4q+3$. Note that $|H(w_1)|=|H(v_g)|+2$.

If $r\le 2q-6$, then $|H(v_g)|\le (q-1)^2-4$, and then $b(H(v_g))=q-1$ by Theorem~\ref{f3}, a contradiction.

If $r=2q-5$, then $|H(v_g)|= (q-1)^{2}-3$, and $(m_2,m_3)=(2,2)$ by our assumption and Theorem~\ref{f3}. Then $m_1\ge 2$ and $a_2-q+1=2$, and then $H(w_1) \in \{P_{m_1+2}+ P_{3}+ P_{1},P_{4}+P_{m_1+1}+ P_{1}\}$ with $|H(w_1)|= (q-1)^{2}-1$, and thus, by Theorem~\ref{f3}, $b(H(w_1))= q-1$, a contradiction.

If $r=2q-4$, then $|H(v_g)|= (q-1)^{2}-2$, and $(m_2,m_3)\in\{(2,2),(3,2)\}$ by our assumption and Theorem~\ref{f3}.
If $m_3=g-2q+1=2$, then $m_2=a_2-q+1\in \{2,3\}$ as $a_1\ge a_2$, which implies $(g,a_{1},a_{2})\in \mathcal C_{2}$, a contradiction. Therefore, $m_3=a_2-q+1$. If $m_2=a_1-q+1$, then $m_2=2$ (otherwise, $a_1-q+1=3$, and then $(g,a_{1},a_{2})\in \mathcal C_{2}$, a contradiction), and then $m_1=(q-1)^{2}-6\ge 3$.
Note that $H(w_1) =P_{m_1+2}+ P_{3}+ P_{1}$ with $|H(w_1)|= (q-1)^{2}$, and hence, by Theorem~\ref{f3}, $b(G)= q-1$, a contradiction. So $m_2=g-2q+1$, furthermore, $g-2q+1=3$ (otherwise, $(g,a_{1},a_{2})\in \mathcal C_{2}$, a contradiction), and thus $H(w_1)=P_{m_1+1}+ P_{5}+ P_{1}$ with $m_1=(q-1)^{2}-7\ge 9$ and $|H(w_1)|=(q-1)^{2}$. By Theorem~\ref{f3}, $b(H(w_1))= q-1$, a contradiction.

If $r=2q-3$, then $|H(v_g)|= (q-1)^{2}-1$, by our assumption and Theorem~\ref{f3}, $(m_1,m_2,m_3)\in J^3$, which implies $(g,a_{1},a_{2})\in \mathcal C_3$, a contradiction.

If $r=2q-2$, then $|H(v_g)|= (q-1)^{2}$, and then by our assumption and Theorem~\ref{f3}, we have $(m_1,m_2,m_3)\in J^4\cup J^5$, which implies $(g,a_{1},a_{2})\in \mathcal C_{4}\cup \mathcal C_{5}\cup \mathcal C_{6}$, a contradiction.\end{proof}

By Corollary \ref{unicyclic3} and Lemmas \ref{q2}-\ref{q4}, we have

\begin{theorem}
\label{b=q}
Suppose that $U_{g}^{a_{1},a_{2}}$ $(a_{1}\ge a_{2})$ is a unicyclic graph  of order $n=q^2+r$ ($1\le r\le 2q+1$). If $1\le r\le 2q-2$, $r+1\le g\le 2q-1$, $a_{1}\le q^{2}-\lfloor\frac{g}{2}\rfloor-1$, $(g,a_{1},a_{2})\notin \mathcal B$ or $2q\le g\le q^{2}$, $a_{2}\geq r-q+1$ and $(g,a_{1},a_{2})\notin \bigcup_{i=1}^6 \mathcal C_i$, then $b(U_{g}^{a_{1},a_{2}})= \left\lceil\sqrt{n}\right\rceil-1=q$.
\end{theorem}

\subsection{Unicyclic graphs in $\mathcal U_{n,g}^{(2)}$ with burning number $q+1$}

In this subsection, we will characterize the graphs in $\mathcal U_{n,g}^{(2)}$ with burning number $q+1$. By Corollary \ref{unicyclic3}, we only need to give some sufficient conditions for a unicyclic graph to satisfy  $k\ge q+1$.

Recall that $(x_1, x_2,\ldots, x_k)$ is an optimal burning sequence for $U_{g}^{a_{1},a_{2}}$, where $n=g+a_{1}+a_{2}$. Suppose that $v_g\in V(T_{k-i_0}(x_{k-i_0}))$ for some $i_0\in [k]$. Then $T_{k-i}(x_{k-i})$ is a path with $|T_{k-i}(x_{k-i})|\le 2i+1$ for any $i\neq i_{0}$.

\begin{lemma}\label{q+1}
If $ r\ge 2q-1$, then $k\ge q+1$.
\end{lemma}

\begin{proof}
Since $ r\ge 2q-1$, we have $n+3=q^2+r+3\geq q^2+2q+2=(q+1)^2+1$. Then $\left\lceil\sqrt{n+3}\right\rceil-1\ge q+1$, and then by (2), $k\ge q+1$.
\end{proof}

By Lemma \ref{q+1}, we can assume $r\le 2q-2$ in the following.

\begin{lemma}
\label{q+1.}
If $3\le g\le r$ or $a_{1}\geq q^{2}-\lfloor\frac{g}{2}\rfloor$, then $k\ge q+1$.
\end{lemma}

\begin{proof}
If $3\le g\le r$, then $a_{1}+a_{2}+1=q^2+r-g+1\geq q^{2}+1$; and if $a_{1}\geq q^{2}-\lfloor\frac{g}{2}\rfloor$, then $a_{1}+\lfloor\frac{g}{2}\rfloor+1\geq q^2+1$. Note that $P_{a_{1}+a_{2}+1}$ and $P_{a_{1}+\lfloor\frac{g}{2}\rfloor+1}$ are isometric subtrees of $U_{g}^{a_{1},a_{2}}$, and hence, by Lemma~\ref{isometric} and Theorem~\ref{Pn}, $k\geq
b(P_{a_{1}+a_{2}+1})=\lceil \sqrt{a_{1}+a_{2}+1}\rceil\ge q+1$ and $k\geq b(P_{a_{1}+\lfloor\frac{g}{2}\rfloor+1})=\left\lceil \sqrt{a_{1}+\lfloor\frac{g}{2}\rfloor+1}\right \rceil \ge q+1$.\end{proof}

\begin{lemma}
\label{q+1.}
If $g\geq q^{2}+1$, then $k\ge q+1$.
\end{lemma}

\begin{proof}  Note that $|T_{k-i_0}(x_{k-i_0})|\le 2i_{0}+1+a_1+a_2$. Hence
$n= \sum_{i=0}^{k-1}(2i+1)+a_{1}+a_{2}=k^2+n-g\le k^2+n-q^2-1,$ that is $k^{2}\geq q^2+1$. Thus $k\ge q+1$.
\end{proof}

\begin{lemma}
\label{q+1.5}
If $2q\le g\le q^{2}$ and $a_{2}\le r-q$, then $k\ge q+1$. Moreover, if $(g,a_{1},a_{2})\in \mathcal C_1$, then $k\ge q+1$.
\end{lemma}

\begin{proof} Since $2q\le g\le q^{2}$ and $a_{2}\le r-q+1$, we have $a_1=q^2+r-g-a_2\ge q-1$. Then $|T_{k-i_0}(x_{k-i_0})|\le 2i_{0}+1+(i_{0}-p_{i_{0}})+a_2$, where $p_{i_{0}}=d(v_g,x_{k-i_0})$. Hence,
$$n=q^2+r=\sum_{i=1}^k|V(T_{i}(x_i))|\le k^2+i_0-p_{i_{0}}+a_2\le k^2+k-1+a_2.\eqno(4)$$
If $a_{2}\le r-q$, then by (4), $ k^2+k\ge q^2+q+1$, i.e., $k\ge q+1$.

If $(g,a_{1},a_{2})\in \mathcal C_1$, then $a_{2}=r-q+1\le q-1$ as $r\le 2q-2$, and then by (4), $n=q^2+r\le k^2+k+r-q$. If $n=q^2+r\le k^2+k-1+r-q$, then $k\ge q+1$. So we assume $n\ge k^2+k+r-q$. Then $n=q^2+r=k^2+k+r-q$, which implies $k=q$
and all inequalities in $(4)$ should be equal, that is, $i_0=k-1$, $p_{i_0}=0$ and $|T_{k-i}(x_{k-i})|=2i+1$ for any $i\not=i_0$. So $x_1=v_g$ and $H(v_g)\cong P_{2}+P_{q^2-2q-1}$, a contradiction.
 \end{proof}

\begin{lemma}
\label{q+1.6}
If $(g,a_{1},a_{2})\in \mathcal B$, then $k\ge q+1$.
\end{lemma}

\begin{proof} Since $(g,a_{1},a_{2})\in \mathcal B$, we have $2q-3\le r=g-1\le 2q-2$ and $a_1\ge a_2= q+1$. Then $|T_{k-i_0}(x_{k-i_0})|\le 2i_{0}+g$, and hence,
$n=q^2+r=\sum_{i=1}^k|T_{i}(x_i)|\le k^2+g-1=k^2+r$. Then $n=q^2+r=k^2+r$(otherwise $n=q^2+r\le k^2+r-1$, then $k\ge q+1$), which implies $k=q$,  $|T_{k-i_0}(x_{k-i_0})|=2i_{0}+g$ and $|T_{k-i}(x_{k-i})|=2i+1$ for any $i\not=i_0$, then $i_0=q-1$ and $|T_1(x_1)|=2i_0+g\in\{4(q-1),4(q-1)+1\}$. By Proposition \ref{Sp4} we can find that $x_1=v_g$, and $T_{1}(x_{1})=U_{g}^{q-1,q-1}$, then $H(v_g)=P_{q^2-2q-1}+P_{2}$, a contradiction with $|T_{k-i}(x_{k-i})|=2i+1$ for any $i\not=i_0$.
\end{proof}

\begin{lemma}
\label{q+1.7}
If $(g,a_{1},a_{2})\in \bigcup_{i=2}^{6}\mathcal C_i$, then $k\ge q+1$.
\end{lemma}

\begin{proof} If $n=q^2+r\le k^2+r-1$, then $k\ge q+1$. So we can assume $n=q^2+r\geq k^2+r$, Then $k=q$ as $k\ge q$. Let $G^*:=U_{g}^{a_{1},a_{2}}-V(T_{1}(x_{1}))$, then by Theorem \ref{b-sequ}, $b(G^*)\le q-1$.

By Proposition \ref{Sp4}, $|T_{k-i_0}(x_{k-i_0})|\le 2i_{0}+1+(s-2)(i_{0}-p_{i_{0}})$ with $s\le4$ and $p_{i_{0}}=d(v_g,x_{k-i_0})$. Hence, by (1),
$$n=q^2+r\le k^2+2(i_0-p_{i_{0}})\le k^2+2k-2=q^2+ 2q-2.\eqno(5)$$

If $(g,a_{1},a_{2})\in \bigcup_{i=4}^{6}\mathcal C_i$,  i.e., $n=q^2+ 2q-2$, then all inequalities in (5) should be equal, that is $p_{i_0}=0$, $i_0=q-1$, $|T_{1}(x_{1})|=4(q-1)+1$, and then $x_1=v_g$, $T_{1}(x_{1})-v_g=4P_{q-1}$ by Proposition \ref{Sp4}, which implies $G^*= P_{m_1}+P_{m_2}+P_{m_3}$ and $|G^*|=(q-1)^2$, where $\{m_1,m_2,m_3\}=\{g-2q+1,a_1-q+1,a_2-q+1\}$ and $m_1\ge m_2\ge m_3\ge 1$. On the other hand, since $(g,a_{1},a_{2})\in \bigcup_{i=4}^{6}\mathcal C_i$, we can check that $(m_1,m_2,m_3)\in J^4\cup J^5$, and thus, by Theorem~\ref{f3}, $b(G^*)=q$, a contradiction.

If $(g,a_{1},a_{2})\in \mathcal C_{3}$, then $n=q^2+2q-3$, and then by (5), $p_{i_0}=0$ (that is $x_1=v_g$) and $i_0=q-1$, furthermore, $|T_{1}(x_{1})|=4(q-1)$ and $T_{1}(x_{1})-v_g=3P_{q-1}+P_{q-2}$ by Proposition \ref{Sp4}, or $|T_{1}(x_{1})|=4(q-1)+1$ and $T_{1}(x_{1})-v_g=4P_{q-1}$, which implies $G^*= P_{m}+P_{m'}+P_{M}$ for some $M\ge m'\ge m\ge 1$. In the former case, $G^*\in \{P_{g-2q+2}+P_{a_1-q+1}+P_{a_2-q+1}, P_{g-2q+1}+P_{a_1-q+2}+P_{a_2-q+1}, P_{g-2q+1}+P_{a_1-q+1}+P_{a_2-q+2}\}$, $|G^*|=(q-1)^2$ and we can check that $(M,m',m)\in J^4\cup J^5$. In the later case, $G^*= P_{g-2q+1}+P_{a_1-q+1}+P_{a_2-q+1}$, $|G^*|=(q-1)^2-1$ and $(M,m',m)\in J^3$. Then in either case, $b(G^*)=q$ by Theorem~\ref{f3}, a contradiction.

If $(g,a_{1},a_{2})\in \mathcal C_{2}$, then $n=q^2+2q-4$, and then by (5), $4q-5\le |T_{1}(x_{1})|\le 4q-3$. If $|T_{1}(x_{1})|=4q-3$, then $x_1=v_g$, $T_{1}(x_{1})-v_g=4P_{q-1}$ by Proposition \ref{Sp4}, which implies $G^*\in \{ P_{q^2-2q-6}+P_3+P_2,P_{q^2-2q-5}+P_2+P_2\}$ and $|G^*|=(q-1)^2-2$, and then $b(G^*)=q$ by Theorem~\ref{f3}, a contradiction. If $|T_{1}(x_{1})|=4(q-1)$, then $x_1=v_g$, $T_{1}(x_{1})-v_g=3P_{q-1}+P_{q-2}$ by Proposition \ref{Sp4}, which implies $G^*\in \{ P_{q^2-2q-5}+P_3+P_2,P_{q^2-2q-6}+P_4+P_2,P_{q^2-2q-6}+P_3+P_3, P_{q^2-2q-4}+P_2+P_2\}$ and $|G^*|=(q-1)^2-1$, and then $b(G^*)=q$ by Theorem~\ref{f3}, a contradiction.
So $|T_{1}(x_{1})|=4q-5$. Then $|G^*|=(q-1)^2$ and $p_{i_{0}}=d(v_g,x_{1})\le 1$. If $x_1=v_g$, then $T_{1}(x_{1})-v_g\in \{3P_{q-1}+P_{q-3}, 2P_{q-1}+2P_{q-2}\}$ by Proposition \ref{Sp4}, which implies $G^*\in \{ P_{q^2-2q-6}+P_4+P_3,P_{q^2-2q-5}+P_4+P_2,P_{q^2-2q-6}+P_5+P_2, P_{q^2-2q-3}+P_2+P_2, P_{q^2-2q-4}+P_3+P_2, P_{q^2-2q-6}+P_5+P_2, P_{q^2-2q-5}+P_3+P_3\}$, and then $b(G^*)=q$ by Theorem~\ref{f3}, a contradiction. So $d(x_1,v_g)=1$, then $x_1\in\{w_1,u_1,v_1\}$ and $T_{1}(x_{1})-\{v_g,x_1\}=3P_{q-2}+P_{q-1}$. If $x_1=w_1$, then $G^*\in \{ P_{q^2-2q-4}+P_4+P_1,P_{q^2-2q-5}+P_4+P_2\}$; if $x_1=u_1$, then $G^*\in \{ P_{q^2-2q-4}+P_3+P_2,P_{q^2-2q-6}+P_4+P_3,P_{q^2-2q-7}+P_4+P_4\}$; and if $x_1=v_1$, then $G^*\in \{ P_{q^2-2q-4}+P_3+P_2,P_{q^2-2q-6}+P_4+P_3,P_{q^2-2q-5}+P_4+P_2\}$, and thus $b(G^*)=q$ by Theorem~\ref{f3}, a contradiction.
\end{proof}

By Corollary \ref{unicyclic3} and Lemmas \ref{q+1}-\ref{q+1.7}, we have

\begin{theorem}
\label{b=q+1..}
Suppose that $U_{g}^{a_{1},a_{2}}$ ($a_{1}\ge a_{2}$) is a unicyclic graph  of order $n=q^2+r$ with $1\le r\le 2q+1$. If $r\geq 2q-1$ or $3\le g\le r$ or $g\geq q^{2}+1$ or $a_{1}\geq q^{2}-\lfloor\frac{g}{2}\rfloor$, or $2q\le g\le q^{2}$ and $a_{2}\le r-q$,
or $(g,a_{1},a_{2})\in \bigcup_{i=1}^{6}\mathcal C_{i}\cup \mathcal B$, then $b(U_{g}^{a_{1},a_{2}})= \left\lceil\sqrt{n}\right\rceil=q+1$.
\end{theorem}


\noindent{\bf Note 2. } By Theorems \ref{b=q} and \ref{b=q+1..}, $b(U_{g}^{a_1,a_2})$ are listed in Table 2, where $g+a_1+a_2=q^2+r$ and $1\le r\le 2q+1$.



\begin{table}[!hbp]
\centering
\vskip.2cm
\begin{tabular}{|c|c|c|c|}
\hline
\hline
$r$ & $g$  & $a_{1},a_{2}$  & $b(U_{g}^{a_{1},a_{2}})$ \\
\hline

\multirow{9}{*}{$1\le r\le 2q-2$} & $3\le g\le r$  & $$ &$q+1$\\
\cline{2-4}
\multirow{9}{*}{} & \multirow{4}{*}{$r+1\le g\le 2q-1$} & $a_{1}\geq q^2-\lfloor\frac{g}{2}\rfloor$~or~ &$\multirow{2}{*}{$q+1$}$\\
\multirow{9}{*}{} & \multirow{4}{*}{} & $(g,a_{1},a_{2})\in \mathcal B$  &$$\\
\cline{3-4}
\multirow{9}{*}{} & \multirow{4}{*}{} & $a_{1}\le q^2-\lfloor\frac{g}{2}\rfloor-1$~and~ &$\multirow{2}{*}{$q$}$\\
\multirow{9}{*}{} & \multirow{4}{*}{} & $(g,a_{1},a_{2})\notin \mathcal B$  &$$\\
\cline{2-4}
\multirow{9}{*}{} & \multirow{4}{*}{$2q\le g\le q^{2}$} & $a_{2}\le r-q$~or~ &$\multirow{2}{*}{$q+1$}$\\
\multirow{9}{*}{} & \multirow{4}{*}{} & $(g,a_{1},a_{2})\in \bigcup_{i=1}^6 \mathcal C_i$ &$$\\
\cline{3-4}
\multirow{9}{*}{} & \multirow{4}{*}{}  & $a_{2}\geq r-q+1$~and~ &\multirow{2}{*}{$q$}\\
\multirow{9}{*}{} & \multirow{4}{*}{}  & $(g,a_{1},a_{2})\notin \bigcup_{i=1}^6 \mathcal C_i$ &$$\\
\cline{2-4}

\multirow{9}{*}{} & $g\geq q^{2}+1$  & $$ &$q+1$\\

\hline
$2q-1\le r\le 2q+1$ & \multicolumn{2}{|c|}{} &$q+1$\\
\hline

\end{tabular}
\caption{Burning number of $U_{g}^{a_{1},a_{2}}$}
\end{table}

\end{document}